\documentclass[11pt,reqno]{amsart}
\usepackage{amsmath,amsthm,amssymb,amscd,url,enumerate,mathtools}
\usepackage{graphicx}
\usepackage[margin=1.5in]{geometry}
\usepackage{afterpage}
\usepackage{tikz}
\usepackage{todonotes}
\usepackage[all]{xy}
\usepackage[colorlinks=true,citecolor={blue},linkcolor = {purple}]{hyperref} 
\usepackage{tabularx}
\usepackage{tabu}

\usepackage{caption}

\usepackage{dutchcal}

\usepackage[square]{natbib}
\setcitestyle{numbers}

\newcommand{\TITLE}{The Monogeneity of Kummer Extensions and Radical Extensions}
\newcommand{\TITLERUNNING}{}

\theoremstyle{plain}
\newtheorem{theorem}{Theorem}

\newtheorem{proposition}[theorem]{Proposition}
\newtheorem{lemma}[theorem]{Lemma}
\newtheorem{corollary}[theorem]{Corollary}
\newtheorem{porism}[theorem]{Porism}

\theoremstyle{definition}

\theoremstyle{remark}
\newtheorem{remark}[theorem]{Remark}

\newtheorem{example}[theorem]{Example}

\numberwithin{theorem}{section}

  {\end{list}}

  {\end{list}}

\newcommand{\tightoverset}[2]{%
  \mathop{#2}\limits^{\vbox to -.5ex{\kern-1.05ex\hbox{$#1$}\vss}}}

\numberwithin{equation}{section} 

\newcommand{\gp}{{\mathfrak{p}}}

\newcommand{\gl}{{\mathfrak{l}}}

\def\Ocal{{\mathcal O}}

\def\pcal{\mathcal{p}}

\newcommand{\FF}{\mathbb{F}}

\newcommand{\QQ}{\mathbb{Q}}

\newcommand{\ZZ}{\mathbb{Z}}

\newcommand{\dnd}{\nmid}

\newcommand{\nth}[1]{{#1}^\text{th}}

\title[\TITLERUNNING]{\vspace*{-1.3cm} \TITLE}

\date{\today}

\author[Hanson Smith]{Hanson Smith}
\address{%
Department of Mathematics, University of Colorado,
Campus Box 395, Boulder, Colorado 80309-0395}
\email{hanson.smith@colorado.edu or hansonsmith101@gmail.com}

\keywords{Monogenic, Power Integral Basis, Relative Integral Basis, Kummer Extension, Radical Extension}
\subjclass[2010]{11R04, 11R18, 11R20}

\thanks{The author would like to thank Sebastian Bozlee, Keith Conrad, and Katherine Stange. The author is especially grateful to Keith for the idea to look into more general radical extensions. The author would also like to thank Anuj Jakhar and the anonymous referee for noticing a mistake in the original proof of Theorem \ref{radical}}

\begin{document}

\sloppy 

\begin{abstract}
We give necessary and sufficient conditions for the Kummer extension $K:=\mathbb{Q}\left(\zeta_n,\sqrt[n]{\alpha}\right)$ to be monogenic over $\mathbb{Q}(\zeta_n)$ with $\sqrt[n]{\alpha}$ as a generator, i.e., for $\mathcal{O}_K=\mathbb{Z}\left[\zeta_n\right]\left[\sqrt[n]{\alpha}\right]$. We generalize these ideas to radical extensions of an arbitrary number field $L$ and provide necessary and sufficient conditions for $\sqrt[n]{\alpha}$ to generate a power $\mathcal{O}_L$-basis for $\mathcal{O}_{L\left(\sqrt[n]{\alpha}\right)}$. 

We also give sufficient conditions for $K$ to be {\bf non}-monogenic over $\mathbb{Q}$ and establish a general criterion relating ramification and relative monogeneity. Using this criterion, we find a necessary and sufficient condition for a relative cyclotomic extension of degree $\phi(n)$ to have $\zeta_n$ as a monogenic generator.
\end{abstract}

\maketitle

\section{Results and Previous Work}
Let $L$ be a number field. We will always denote the ring of integers by $\Ocal_L$. Suppose $M$ is a finite extension of $L$. If $\Ocal_M=\Ocal_L[\theta]$ for an algebraic integer $\theta\in M$, then we say \emph{$M$ is monogenic over $L$} or $\Ocal_M$ has a \emph{power $\Ocal_L$-basis.} We note that in general $\Ocal_M$ may not be free over $\Ocal_L$; however, monogeneity implies freeness. When $L$ is $\QQ$ we will simply say $M$ is \emph{monogenic} or $\Ocal_M$ has a \emph{power integral basis.} 

Suppose for the moment that $L$ is a number field containing a primitive $n^{\text{th}}$ root of unity, $\zeta_n$. A \emph{Kummer extension of degree $n$} is an extension of the form $L\left(\sqrt[n]{\alpha}\right)$, where $x^n-\alpha$ is irreducible over $L$. The Kummer extensions of $L$ of degree $n$ are exactly the cyclic extensions of $L$ of order $n$. When $L=\QQ\left(\zeta_n\right)$, a Kummer extension will be denoted by $K$. If $L$ is an arbitrary number field (not necessarily containing the $n^{\text{th}}$ roots of unity), we call an extension of the form $L\left(\sqrt[n]{\alpha}\right)$ a \emph{radical extension}\footnote{Radical number fields are also known as `pure.'}. Letting $L$ again be arbitrary, when $n=2k$ with $k$ odd, one has $L(\zeta_n)=L\left(\zeta_k\right)$. For this reason, when we speak of the $n^{\text{th}}$ cyclotomic field or an $n^{\text{th}}$ root of unity, it is often assumed that $n\not \equiv 2 \bmod 4$. Context will make our intent clear.

The main result of this paper is Theorem \ref{radical}, where we describe necessary and sufficient conditions for the ring of integers of the radical extension $L\left(\sqrt[n]{\alpha}\right)$ to have a power $\Ocal_L$-basis generated by $\sqrt[n]{\alpha}$. This result can be illustrated by the important special case of Kummer extensions, which we state below for $n$ an odd prime. In our investigation of Kummer extensions, we also obtain sufficient conditions for when $K$ \textbf{is not} monogenic over $\QQ$; this is stated below as well.


\begin{figure}[h!]
\hspace*{-.5 in}\xymatrix{
& & K:=\QQ\left(\zeta_n\right)\left(\sqrt[n]{\alpha}\right) \ar@{-}[d]_{\ZZ/n\ZZ}\\
& & \QQ\left(\zeta_n\right) \ar@{-}[d]_{\left(\ZZ/n\ZZ\right)^*}\\
& & \QQ}
\caption{Kummer extensions we consider}
\protect{\label{Diamond}}
\end{figure} 


\begin{theorem}\label{KoverQzeta}
Let $p$ be an odd, rational prime, and let $\gp:=(1-\zeta_p)$ be the unique prime of $\ZZ[\zeta_p]$ above $p$. Let $\alpha\in \ZZ[\zeta_p]$, and suppose that $x^p-\alpha$ is irreducible in $\ZZ[\zeta_p][x]$. The ring of integers $\Ocal_{\QQ\left(\zeta_p,\sqrt[p]{\alpha}\right)}$ is $\ZZ[\zeta_p]\left[\sqrt[p]{\alpha}\right]$ if and only if $\alpha$ is square-free as an ideal\footnote{Note that the unit ideal is square-free.} of $\ZZ[\zeta_p]$ and the congruence 
\begin{equation}\label{Wiererichzeta}
\alpha^p\equiv\alpha \bmod {(1-\zeta_p)^2}
\end{equation} 
\textbf{is not} satisfied.
\end{theorem}

\begin{theorem}\label{KoverQ} Suppose there exists a rational prime $l$ such that $l\equiv 1 \bmod n$ and $l<n\cdot\phi(n)$. Let $\alpha \in \ZZ\left[\zeta_n\right]$ be relatively prime to $l$.  Suppose further that $\alpha$ is an $\nth{n}$ power residue modulo some prime of $\ZZ\left[\zeta_n\right]$ above $l$ and that $x^n-\alpha$ is irreducible in $\ZZ\left[\zeta_n\right][x]$. Then $K=\QQ\left(\zeta_n,\sqrt[n]{\alpha}\right)$ \textbf{is not} monogenic over $\QQ$. Moreover, $l$ is an essential discriminant divisor, i.e., $l$ divides $\left[\Ocal_K:\ZZ[\theta]\right]$ for every integer $\theta$ such that $\QQ(\theta)=K$.
\end{theorem}

Theorem \ref{KoverQzeta} stands in marked contrast to the situation over $\QQ$. Gras \cite{Grasprime} shows that the only monogenic abelian extensions of $\QQ$ of prime degree $\geq 5$ are maximal real subfields of cyclotomic fields. In order to obtain a single monogenic abelian extension of prime degree $p\geq 5$, we must have $p=\frac{\phi(n)}{2}$, where $\phi$ is Euler's phi function. Over $\QQ(\zeta_p)$, however, we are able to construct infinitely many monogenic abelian extensions of prime degree $p$.

In addition to the theorems mentioned above, we give a more general criterion relating ramification to relative monogeneity, Proposition \ref{nonmonorelative}. The proof of Proposition \ref{nonmonorelative} serves to highlight our methods. This proposition is then applied to prove Corollary \ref{ZetaNotMono}:  For an arbitrary number field $L$, the ring of integers $\Ocal_{L(\zeta_n)}=\Ocal_L\left[\zeta_n\right]$ if and only if $\gcd\left(n,\Delta_L\right)=1$. We use the classical strategy of Dedekind to prove Theorem \ref{KoverQ}, while our other results are established using a generalization, by Kumar and Khanduja, of Dedekind's index criterion to relative extensions (Theorem \ref{Dedekindindex}).

The outline of the paper is as follows. At the end of this section we will briefly survey the literature regarding the monogeneity of abelian extensions, relative monogeneity, and the monogeneity of radical extensions. Section \ref{background} recalls the necessary tools that we will use. With Section \ref{monoandram}, we state and prove our proposition relating relative monogeneity and ramification. Section \ref{Kummermono} is concerned with the proof of Theorem \ref{KoverQzeta}. This section also serves to illustrate how we will approach the proof of Theorem \ref{radical}. In Section \ref{Kummernonmono}, we prove Theorem \ref{KoverQ}. Finally, Section \ref{genradsection} states and establishes our main result on the monogeneity of radical extensions.

The literature regarding monogenic fields is extensive:  For a nice treatise on monogeneity that focuses on using index form equations, see Ga\'al's book \cite{GaalsBook}. 
With the inclusion of numerous references, this book is likely the most modern and thorough survey of the subject. For another bibliography of monogeneity, see Narkiewicz's text \cite[pages 79-81]{Nark}. Narkiewicz also considers monogeneity in \cite{AlgebraicStory}. Zhang's brief survey \cite{ZhangSurvey}, though unpublished, is a nice overview. 

Investigations into the monogeneity of abelian number fields are classical. For example, the monogeneity of quadratic fields is immediate and the monogeneity of cyclotomic fields was established very early. As mentioned above, Gras \cite{Grasprime} has shown that, with the exception of the maximal real subfields of cyclotomic fields, abelian extensions of $\QQ$ of prime degree greater than or equal to 5 are not monogenic. Generally, Gras \cite{Gras} has shown that almost all abelian extensions of $\QQ$ with degree coprime to 6 are not monogenic. The extensions of $\QQ$ we show are non-monogenic in this paper generally have degree divisible by 2. Previous to Gras, Payan \cite{Payan} found necessary conditions for monogeneity of certain cyclic extensions. Cougnard \cite{Cougnard} builds on the ideas of Payan and establishes more stringent conditions for an imaginary quadratic field to have a monogenic cyclic extension of prime degree. Ichimura \cite{Ichimura} establishes the equivalence of a certain unramified Kummer extension being monogenic over its base field and the Kummer extension being given by the $p^{\text{th}}$ root of a unit of a specified shape. 
Khan, Katayama, Nakahara, and Uehara \cite{NonMonoCyclo} study the monogeneity of the compositum of a cyclotomic field, with odd conductor $n\geq 3$ or even conductor $n\geq 8$ with $4\mid n$, and a totally real number field, distinct from $\QQ$ and with discriminant coprime to the discriminant of the cyclotomic field. They show that no such compositum is monogenic. The monogeneity of the compositum of a real abelian field and an imaginary quadratic field is studied by Motoda, Nakahara, and Shah \cite{MotodaNakaharaShah}. When the conductors are relatively prime and the imaginary quadratic field is not $\QQ(i)$, they establish that monogeneity is not possible. Shah and Nakahara \cite{ShahNakahara} show the monogeneity of certain imaginary index 2 subfields of cyclotomic fields. They also prove a criterion for non-monogeneity in Galois extensions based on the ramification and inertia of a small prime. Jung, Koo, and Shin \cite{JungKooShinWeierstrassUnits} use Weierstrass units to build relative monogenic generators for the composita of certain Ray class fields of imaginary quadratic fields. Motoda and Nakahara \cite{MotodaNakahara} show that if the Galois group of $L$ is elementary 2-abelian and $L$ has degree $\geq 16$, then $L$ is not monogenic over $\QQ$. They also establish partial results in the case that $[L:\QQ]=8$. Chang \cite{Chang} completely describes the monogeneity of the Kummer extension $K$ when $[K:\QQ]=6$. Ga\'{a}l and Remete \cite{GaalRemetenon} investigate $[K:\QQ]=8$. Though we do not outline it further here, there is a wealth of literature on monogenic abelian extensions of a fixed degree. The interested reader should consult the surveys mentioned earlier.

Ga\'{a}l, Remete, and Szab\'{o} \cite{MR3513574} study the relation between absolute monogeneity, i.e. monogeneity over $\QQ$, and relative monogeneity. Suppose $L$ is a number field, $\Ocal_L=\ZZ[\theta]$, and $R$ is an order of a subfield of $L$. They establish that $\theta$ can always be used to construct a power $R$-integral basis for $\Ocal_L$. Relative power integral bases are also studied in \cite{CubicRelative}, \cite{QuarticRelative}, \cite{GaalRemete0}, \cite{GaalRemeteMonoRelative}, and \cite{QuartQuadRelative}.

Radical extensions are also a classical object of study. In 1910, Westlund \cite{Westlund} computed the discriminant and an integral basis for the radical extensions $\QQ\left(\sqrt[p]{\alpha}\right)$ over $\QQ$, where $\alpha\in \ZZ$ and $p$ is a prime. Westlund also identified when $\sqrt[p]{\alpha}$ yields a power integral basis for $\QQ\left(\sqrt[p]{\alpha}\right)$. Using Dedekind's index criterion (Theorem \ref{Dedind}) and the Montes algorithm, Gassert \cite{Alden} gives necessary and sufficient conditions for the ring of integers of $\QQ\left(\sqrt[n]{\alpha}\right)$ to be $\ZZ\left[\sqrt[n]{\alpha}\right]$. Having $\sqrt[n]{\alpha}$ generate a power integral basis is dependent on the congruence 
\begin{equation}\label{Wieferich}
\alpha^p\equiv \alpha \bmod {p^2},
\end{equation}
where $p$ divides $n$. Loosely speaking, non-zero solutions to Congruence \eqref{Wieferich} are obstructions to $\sqrt[n]{\alpha}$ generating a power integral basis. A prime $p$ for which Congruence \eqref{Wieferich} has a solution is called a \emph{Wieferich prime\footnote{Wieferich \cite{Wieferich} studied these primes, with $\alpha=2$, in relation to Fermat's Last Theorem.} to the base $\alpha$}. See Keith Conrad's excellent expository note \cite{KCNoteRadical} for background on $\ZZ$-power bases of radical extensions and the history of Wieferich primes.

The monogeneity of radical extensions of a given degree has been studied extensively. The radical quartic case is investigated by Funakura, who finds infinitely many monogenic fields \cite{Funakura}. Ga\'{a}l and Remete \cite{GaalRemete}, characterize the only power integral bases of a number of infinite families of radical quartic fields using binomial Thue equations and extensive calculations on a supercomputer.  In \cite{AhmadNakaharaHameed} and \cite{AhmadNakaharaHusnineMono} degree six is studied. Degree eight is considered in \cite{HameedNakaharaBases} and \cite{HameedNakaharaHusnineAhmad}. Degree equal to a power of 2 is also studied in \cite{HameedNakaharaHusnineAhmad}. For degree $n$ with $3\leq n\leq 9$, Ga\'al and Remete \cite{GaalRemetePeriodic} establish a periodic characterization of integral bases.


One can also ask about the extent to which monogeneity can fail. First a few definitions:  A quantity related to monogeneity is the field index. The \emph{field index} is defined to be the pair-wise greatest common divisor $\gcd_{\alpha\in \Ocal_K} \left[\Ocal_K:\ZZ[\alpha]\right]$. Note that $K$ can have field index 1 and still not be monogenic; see Example \ref{gcdindices=1}. Define the \emph{minimal index} to be $\min_{\alpha\in \Ocal_K} \left[\Ocal_K:\ZZ[\alpha]\right]$. Monogeneity is equivalent to having minimal index equal to 1. An early result of Hall \cite{Hall} shows that there exist cubic fields with arbitrarily large minimal indices. In \cite{SpearmanYangYoo}, this is generalized to show that every cube-free integer occurs as the minimal index of infinitely many radical cubic fields. Monogeneity is equivalent to requiring exactly one ring generator; Pleasants \cite{Pleasants} shows that the number of generators needed for a field of degree $n$ is less than $\ln (n)/\ln(2)+1$, and, if 2 splits completely, the minimal number of generators is $\lfloor \ln(n)/\ln(2)+1\rfloor$.

\section{Background and Necessary Lemmas}\label{background}

\textbf{Notation:}  An overline always denotes reduction modulo a prime. A $\Delta$ denotes a discriminant, and a subscript on $\Delta$ indicates the object whose discriminant we are considering. A subscript is also used to indicate localization. For example, $\left(\Ocal_L\right)_{\gp}$ is $\Ocal_L$ localized at $\gp$. A choice of uniformizer is indicated by $\pi$ with the ideal of localization in the subscript. In the aforementioned context, $\pi_{\gp}$ is a uniformizer. The normalized valuation associated with a prime $\gp$ is denoted by $v_{\gp}$.

We start with some ideas of Dedekind based on work of Kummer. The following is often called Dedekind's criterion and first appeared in \cite{Dedekind}. Since we have two criteria due to Dedekind, we will call the following Dedekind's criterion for splitting. 

\begin{theorem}\label{DedCrit}
Let $f(x)\in \ZZ[x]$ be monic and irreducible, let $\theta$ be a root, and let $L=\QQ(\theta)$ be the number field generated by $\theta$. If $p\in \ZZ$ is a prime that does not divide $[\Ocal_L:\ZZ[\theta]]$, then the factorization of $p$ in $\Ocal_L$ mirrors the factorization of $f(x)$ modulo $p$. That is, if
\[f(x)\equiv \varphi_1(x)^{e_1}\cdots \varphi_r(x)^{e_r} \bmod p\]
is a factorization of $\overline{f(x)}$ into irreducibles in $\FF_p[x]$, then $p$ factors into primes in $\Ocal_L$ as 
\[p=\gp_1^{e_1}\cdots \gp_r^{e_r}.\]
Moreover, the residue class degree of $\gp_i$ is equal to the degree of $\varphi_i$.
\end{theorem}
An expository proof can be found in many algebraic number theory texts. For example, see \cite[Proposition 4.33]{Nark}. We note there is a natural generalization to relative extensions of number fields. See \cite[Chapter I, Theorem 7.4]{Janusz}.

Using this criterion, Dedekind was the first to demonstrate a number field that was not monogenic. Dedekind considered the cubic field generated by a root of $x^3-x^2-2x-8$. He showed that the prime 2 splits completely. If there were a possible power integral basis, then one would be able to find a cubic polynomial, generating the same number field, that splits completely into distinct linear factors modulo 2. Since there are only two distinct linear polynomials in $\FF_2[x]$, this is impossible. Hence the number field cannot be monogenic. More generally, if a prime $p<n$ splits completely in an extension $L/\QQ$ of degree $n$, then $L$ is not monogenic. We will use the same strategy as Dedekind to construct non-monogenic fields. Hensel \cite{Hensel1894} built on these ideas to show the following.

\begin{theorem}\label{Hensel} Fix a prime $p$. The prime $p$ divides $\left[\Ocal_L:\ZZ[\theta]\right]$ for every algebraic integer $\theta$ generating $L$ over $\QQ$ if and only if there is an integer $f$ such that the number of prime ideal factors of $p\Ocal_L$ with inertia degree $f$ is greater than the number of monic irreducibles of degree $f$ in $\FF_p[x]$.
\end{theorem}

Any $p$ satisfying Theorem \ref{Hensel} is called an \emph{essential discriminant divisor}\footnote{Confusingly, essential discriminant divisors are sometimes called \emph{inessential discriminant divisors} in the literature. We prefer \emph{essential discriminant divisor} because, for a root of a polynomial to generate the number field in question, it is essential that $p$ divide the polynomial's discriminant.} or a \emph{common index divisor}. It turns out that essential discriminant divisors are not the only obstruction to monogeneity:

\begin{example}\label{gcdindices=1}\cite[Chapter 2.2.6]{Nark} Consider the number field given by $L = \QQ\left(\sqrt[3]{7\cdot 5^2}\right)=\QQ\left(\sqrt[3]{5\cdot 7^2}\right)$. The elements $\left\{1, \sqrt[3]{7\cdot 5^2}, \sqrt[3]{7^2\cdot 5}\right\}$ form an integral basis of $L$. For any fixed prime $p$, one can find $\theta\in \Ocal_L$ such that $\left[\Ocal_L:\ZZ[\theta]\right]$ is not divisible by $p$; however, $L$ is not monogenic.
\end{example}

We will use another criterion of Dedekind, which we'll call Dedekind's index criterion, to establish monogeneity. First, we state the version Dedekind proved, with $\QQ$ as the base field.

\begin{theorem}\label{Dedind} Let $f(x)$ be a monic, irreducible polynomial in $\ZZ[x]$, $\theta$ a root of $f$, and $L=\QQ(\theta)$. If $p$ is a rational prime, we have
\[f(x)\equiv \prod_{i=1}^rf_i(x)^{e_i} \bmod p,\]
where the $f_i(x)$ are monic lifts of the irreducible factors of $\overline{f(x)}$ to $\ZZ[x]$. Define
\[d(x):=\dfrac{f(x)-\prod\limits_{i=1}^r f_i(x)^{e_i}}{p}.  \]
Then $p$ divides $\left[\Ocal_L:\ZZ[\theta]\right]$ if and only if $\gcd\left(\overline{f_i(x)}^{e_i-1},\overline{d(x)}\right)\neq 1$ for some $i$, where we are taking the greatest common divisor in $\FF_p[x]$.
\end{theorem}

Recently, Kumar and Khanduja, using completely different methods from those of Dedekind, have proved a generalization of Dedekind's index criterion to relative extensions. This generalization will be very useful to us.

\begin{theorem} \cite[Theorem 1.1]{KandK}\label{Dedekindindex}
Let $R$ be a Dedekind domain with quotient field $L$, and let $f(x)$ be a monic, irreducible polynomial in $R[x]$ with $\theta$ a root. Define $M=L(\theta)$, and suppose $\gp$ is a prime of $R$. We have 
\[f(x)\equiv \prod_{i=1}^r f_i(x)^{e_i} \bmod \gp,\]
where the $f_i(x)$ are monic lifts of the irreducible factors of $\overline{f(x)}$ to $R[x]$. Note the integral closure of $R_\gp$ in $M$ is $\left(\Ocal_M\right)_\gp$. Define the polynomial $d(x)\in R_\gp[x]$ to be
\[d(x):=\dfrac{f(x)-\prod\limits_{i=1}^r f_i(x)^{e_i}}{\pi_{\gp}}.\]
Then $\left(\Ocal_M\right)_\gp=R_\gp[\theta]$ if and only if $\overline{f_i(x)}^{e_i-1}$ is coprime to $\overline{d(x)}$ for each $i$.
\end{theorem}

With Equation \eqref{polyfieldgeneral}, we will see that the conclusion of Theorem \ref{Dedekindindex} is exactly what we need to study $\left[\Ocal_M:\Ocal_L[\theta]\right]$. The interested reader should consult \cite{3Criteria} for a nice discussion of and comparison between three different criteria for monogeneity:  Dedekind's index criterion, a theorem of Uchida \cite{Uchida}, and a theorem of L\"uneburg \cite{Luneburg}. For other, similar generalizations of Dedekind's index criterion see \cite{CharkaniDeajim}, \cite{Ershov}, and, for the greatest generality, \cite{ElFadilBoulagouazDeajim}.

In addition to the work of Dedekind, we will need a few facts about cyclotomic, radical, and Kummer extensions. First, we recall the following well-known formula relating polynomial discriminants and field discriminants. Let $f$ be a monic, irreducible polynomial of degree $n>1$, let $\theta$ be a root, and write $L=\QQ(\theta)$, then
\begin{equation}\label{polyfield}
\Delta_f=\Delta_L[\Ocal_L:\ZZ[\theta]]^2.
\end{equation}

Equation \eqref{polyfield} admits a generalization to relative extensions. We will specialize \cite[Chapter III, Equation 2.4]{FrohlichTaylor} for our purposes. Let $L$ be a number field, and let $M$ be a finite extension of $L$ generated by a root, $\theta$, of a monic, irreducible polynomial $f(x)\in \Ocal_L[x]$. As ideals, we have the equality

\begin{equation}\label{polyfieldgeneral}
\Delta_f = \Delta_{M/L}\left[\Ocal_M:\Ocal_L[\theta]\right]^2.
\end{equation}
Thus, in studying monogeneity, we need only consider the prime factors of $\Delta_f$.

Suppose $M$ and $N$ are two finite extensions of a number field $L$. We call $M$ and $N$ \emph{arithmetically disjoint} (over $L$) if they are linearly disjoint and, as ideals, $\gcd\left(\Delta_{M/L},\Delta_{N/L}\right)=\Ocal_L$. The following is Proposition III.2.13 of \cite{FrohlichTaylor}.

\begin{proposition}\label{arithdisjoint}
If $M$ and $N$ are arithmetically disjoint over $L$, then $\Ocal_{MN}=\Ocal_M\cdot\Ocal_N$ as $\Ocal_L$-modules. 
\end{proposition} 

Proposition \ref{arithdisjoint} will be useful in studying the monogeneity of relative cyclotomic extensions. 

Turning to cyclotomic extensions of $\QQ$, the following is Lemma 6 of Chapter III of \cite{CasselsFrohlich}.

\begin{lemma}\label{cyclodisc}
The discriminant of $\QQ\left(\zeta_n\right)$ over $\QQ$ is
\[\Delta_{\QQ\left(\zeta_n\right)/\QQ}=n^{\phi(n)}\bigg/\prod_{p\mid n}p^\frac{\phi(n)}{p-1}, \] where $\phi$ denotes Euler's phi function. 
Further, an integral basis for $\Ocal_{\QQ\left(\zeta_n\right)}$ is given by 1 and the powers $\zeta_n^k$ with $1\leq k\leq\phi(n)-1$.
\end{lemma}

Lemma \ref{cyclodisc} and Equation \eqref{polyfield} yield the following corollary.

\begin{corollary}\label{cyclopolydisc}
The cyclotomic polynomial $\phi_n(x)$ has discriminant
\[\Delta_{\phi_n}=n^{\phi(n)}\bigg/\prod_{p\mid n}p^\frac{\phi(n)}{p-1}.\]
\end{corollary}

It is useful to understand the splitting of primes in cyclotomic extensions. 

\begin{lemma}\label{cyclosplit} \cite[III.1 Lemma 4]{CasselsFrohlich}:  If $p$ is a prime not dividing $n$, then it is unramified in $\QQ\left(\zeta_n\right)$ and its residue class degree is the least positive integer $f$ such that $p^f\equiv 1 \bmod n$.
\end{lemma}

Bringing our attention to radical and Kummer extensions, consider the polynomial $x^n-\alpha$. One computes
\begin{equation}\label{discf}
\Delta_{x^n-\alpha}=\left(-1\right)^\frac{n^2-n}{2}n^n(-\alpha)^{n-1}.
\end{equation}
One can also derive this by specializing Theorem 4 of \citep{DiscTri}. 

The following describes splitting in Kummer extensions.

\begin{lemma}\label{Kummersplit} \cite[III.2 Lemma 5]{CasselsFrohlich}:  The discriminant of $K=\QQ\left(\zeta_n,\sqrt[n]{\alpha}\right)$ over $\QQ\left(\zeta_n\right)$ divides $n^n\alpha^{n-1}$. A prime $\gp$ of $\ZZ\left[\zeta_n\right]$ is unramified in $K$ if $\gp\dnd n\alpha$. In this case, the residue class degree of $\gp$ is the least positive integer $f$ such that $\alpha^f\equiv x^n \bmod \gp$ is solvable.
\end{lemma}

\section{Monogeneity and Ramification}\label{monoandram}



In this section we present a proposition relating monogeneity and ramification. The result is likely classical, but we include it here to highlight our methods. 

\begin{proposition}\label{nonmonorelative}
Let $L$ be a number field, $h(x)$ a monic, irreducible polynomial in $\Ocal_L[x]$, and $\eta$ a root of $h(x)$. Suppose $\gp$ is a prime of $L$ above the rational prime $p$ such that $\gp \mid \Delta_h$. Let $M$ be an extension of $L$ such that $h(x)$ is irreducible in $M$. If $\gp$ is ramified in $M$, then $p\mid [\Ocal_{M(\eta)}:\Ocal_M[\eta]]$.
\end{proposition}
 
The setup of Proposition \ref{nonmonorelative} is summarized in Figure \ref{notmono}.
\begin{figure}[h!]
\centering
\hspace*{.01 in}
\xymatrix{
 M(\eta) \ar@{-}[d]^{\text{not monogenic via } \eta} & \\
 M \ar@{-}[d] & \pcal^{e>1} \prod\limits_i\pcal_i^{e_i} \ar@{-}[d] \\
 L & \gp }
\caption{Diagram for Proposition \ref{nonmonorelative}}
\protect{\label{notmono}}
\end{figure}
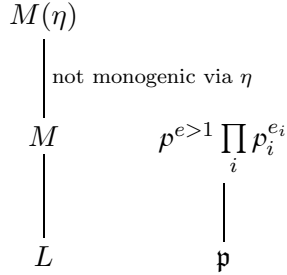

\begin{proof} We will use Theorem \ref{Dedekindindex} to show that $p$ divides $\left[\Ocal_{M(\eta)}:\Ocal_M[\eta]\right]$. Reducing $h(x)$ modulo $\gp$ and choosing lifts of the irreducible factors to $\Ocal_L[x]$, we have
\begin{equation}\label{modulogp}
h(x)\equiv h_0(x)^{e_0} h_1(x) \bmod \gp,
\end{equation}
where $e_0>1$. Such an $h_0$ exists since $\gp \mid \Delta_h$. 

Let $\pcal$ be a prime of $M$ that is ramified above $\gp$. Consider the element of $\left(\Ocal_M\right)_{\pcal}[x]$ given by
\[d(x) = \frac{h(x)-h_0(x)^{e_0} h_1(x)}{\pi_{\pcal}}.\]
Let $\eta_0$ be a root of $h_0$ in some extension of $\left(\Ocal_M\right)_{\pcal}$. 
For $\eta$ to yield a power $\Ocal_M$-basis, it is necessary that $d(\eta_0)\not\equiv 0 \bmod \pi_{\pcal}^2$. Equation \eqref{modulogp} shows that $d(\eta_0)\equiv 0 \bmod \pi_{\gp}$. Since $\pi_{\pcal}^2\mid \pi_{\gp}$,  we see $d(\eta_0)\equiv 0 \bmod \pi_{\pcal}^2$ and $p\mid \left[\Ocal_{M(\eta)}:\Ocal_M[\eta]\right]$.
\end{proof}

Proposition \ref{nonmonorelative} sheds some light on the monogeneity of cyclotomic relative extensions:

\begin{corollary}\label{ZetaNotMono}
Let $L$ be a number field in which the $n^{\text{th}}$ cyclotomic polynomial is irreducible, where $n>2$ is any integer not congruent to $2$ modulo $4$. Then $\Ocal_L\left[\zeta_n\right]=\Ocal_{L\left(\zeta_n\right)}$ if and only if $\gcd\left(n,\Delta_L\right)= 1$.
\end{corollary}

\begin{proof}
The contrapositive is given by Proposition \ref{nonmonorelative}. If $\gcd\left(n,\Delta_L\right)= 1$, then $L$ and $\QQ\left(\zeta_n\right)$ are arithmetically disjoint over $\QQ$, and the result follows from Proposition \ref{arithdisjoint}. One can also prove this direction via a computation with Theorem \ref{Dedekindindex}.
\end{proof}

We can contrast the above Proposition \ref{nonmonorelative} to the following example.
\begin{example}\label{quads}
Let $k,m\in \ZZ$ with $\gcd(k,m)=1$, $k$ and $m$ square-free, $m\equiv 1 \bmod 4$, and $k\equiv 2,3 \bmod 4$. One can use Theorem \ref{Dedekindindex} to show that a $\ZZ\left[\dfrac{1+\sqrt{m}}{2}\right]$-basis of $\Ocal_{\QQ\left(\sqrt{m},\sqrt{k}\right)}$ is given by $1$ and $\sqrt{k}$. Thus, in this case, a root of a polynomial in $\ZZ[x]$ yields a power $\ZZ\left[\dfrac{1+\sqrt{m}}{2}\right]$-basis for $\Ocal_{\QQ\left(\sqrt{m},\sqrt{k}\right)}=\ZZ\left[\dfrac{1+\sqrt{m}}{2},\sqrt{k}\right]$. The interested reader should consult \cite{RelIntegralQuartic} for an in-depth study of relative integral bases of quartic fields with quadratic subfields.
\end{example}

\section{Monogeneity of $K$ over $\QQ(\zeta_p)$}\label{Kummermono}

We wish to establish Theorem \ref{KoverQzeta}:  \textit{Let $p$ be an odd, rational prime, and let $\gp:=(1-\zeta_p)$ be the unique prime of $\ZZ[\zeta_p]$ above $p$. Let $\alpha\in \ZZ[\zeta_p]$, and suppose that $x^p-\alpha$ is irreducible in $\ZZ[\zeta_p][x]$. The ring of integers $\Ocal_{\QQ\left(\zeta_p,\sqrt[p]{\alpha}\right)}$ is $\ZZ[\zeta_p]\left[\sqrt[p]{\alpha}\right]$ if and only if $\alpha$ is square-free as an ideal of $\ZZ[\zeta_p]$ and the congruence 
\begin{equation}\label{Wiererichzeta2}
\alpha^p\equiv\alpha \bmod {(1-\zeta_p)^2}
\end{equation} 
\textbf{is not} satisfied.}

Note that Congruence \eqref{Wiererichzeta2} is exactly the Wieferich congruence, Congruence \eqref{Wieferich}, but with respect to the prime $(1-\zeta_p)$. 
We will see that the analogue of Congruence \eqref{Wieferich} in Theorem \ref{radical} is a bit more complicated. This is due to the potential for higher powers of a prime to divide $n$ and the need to accommodate arbitrary residue class degrees.

\begin{proof}
Recall that $\Delta_{x^p-\alpha}=\left(-1\right)^\frac{p^2-p}{2}p^p(-\alpha)^{p-1}$. Equation \ref{polyfieldgeneral} and the discussion afterwards show that for questions of monogeneity, we need only consider the prime divisors of $\Delta_{x^p-\alpha}$. We will contend with the prime divisors of $\alpha$, then we will contend with $\gp$. In both cases, we will use Theorem \ref{Dedekindindex}.

Suppose $\mathfrak{l}$ is a prime of $\ZZ[\zeta_p]$ dividing $\alpha$. The reduction of $x^p-\alpha$ modulo $\mathfrak{l}$ is $\overline{x}^p$. So, in the notation of Theorem \ref{Dedekindindex}, we have
\[d(x)=\dfrac{x^p-\alpha-x^p}{\pi_{\gl}}=\frac{-\alpha}{\pi_{\gl}}.\]
Now $v_{\gl}\left(\frac{-\alpha}{\pi_{\gl}}\right)=0$ if and only if $v_{\gl}(\alpha)=1$. If $v_{\gl}(\alpha)=1$, the reduction $\overline{\frac{-\alpha}{\pi_{\gl}}}$ generates the unit ideal. In particular, $\overline{\frac{-\alpha}{\pi_{\gl}}}$ is relatively prime to $\overline{x}^{p-1}$. Conversely, if $v_{\gl}(\alpha)>1$, then $\overline{\frac{-\alpha}{\pi_{\gl}}}=0$ and is not relatively prime to $\overline{x}^{p-1}$. With Theorem \ref{Dedekindindex}, we see \[\left(\Ocal_K\right)_{\gl}=\left(\ZZ[\zeta_p]\right)_{\gl}\left[\sqrt[p]{\alpha}\right]\]
if and only if $v_{\gl}(\alpha)=1$.

Next, we contend with $\gp$. We localize $\ZZ[\zeta_p]$ at $\gp$ and choose $1-\zeta_p$ to be the uniformizer. The reduction of $x^p-\alpha$ modulo $\gp$ is $\left(\overline{x-\alpha}\right)^p$.
We have
\[d(x)=\dfrac{x^p-\alpha-(x-\alpha)^p}{1-\zeta_p}.\]
Evaluating at $\alpha$, we see that $\overline{d(x)}$ is relatively prime to $\overline{x-\alpha}$ if and only if 
\[\alpha^p\not\equiv\alpha \bmod {(1-\zeta_p)^2}.\]
Applying Theorem \ref{Dedekindindex}, our result follows. Note, our argument here does not depend on whether or not $\gp$ divides $\alpha$.
\end{proof}


\section{Non-monogeneity of $K$ over $\QQ$}\label{Kummernonmono}

In this section, we will prove Theorem \ref{KoverQ}:  \textit{Suppose there exists a rational prime $l$ such that $l\equiv 1 \bmod n$ and $l<n\cdot\phi(n)$. Let $\alpha \in \ZZ\left[\zeta_n\right]$ be relatively prime to $l$.  Suppose further that $\alpha$ is an $\nth{n}$ power residue modulo some prime of $\ZZ\left[\zeta_n\right]$ above $l$ and that $x^n-\alpha$ is irreducible in $\ZZ\left[\zeta_n\right][x]$. Then $K=\QQ\left(\zeta_n,\sqrt[n]{\alpha}\right)$ \textbf{is not} monogenic over $\QQ$. Moreover, $l$ is an essential discriminant divisor, i.e., $l$ divides $\left[\Ocal_K:\ZZ[\theta]\right]$ for every $\theta$ such that $\QQ(\theta)=K$.}

\begin{proof}
We will use Dedekind's method for proving a number field is not monogenic. From Lemmas \ref{cyclosplit} and \ref{Kummersplit}, we see that $l$ splits completely in $K$. If $K$ is monogenic over $\QQ$, then Theorem \ref{DedCrit} shows that the factorization of $l$ in $K$ is mirrored by the factorization of a degree $n\cdot \phi(n)$ polynomial modulo $l$. Thus there is a degree $n\cdot \phi(n)$ polynomial that generates $K$ over $\QQ$ and factors into distinct linear factors modulo $l$. Since $l<n\cdot\phi(n)$, we see this is impossible. Thus $K$ is not monogenic over $\QQ$. Applying Theorem \ref{Hensel}, we see $l$ is in fact an essential discriminant divisor.
\end{proof}

\begin{remark}\label{generalKoverQ} If $k$ denotes the multiplicative order of $l$ modulo $n$, the number of irreducible polynomials in $\FF_l[x]$ of degree $k$ is $\frac{1}{k}\sum_{d\mid k}\mu\left(\frac{k}{d}\right)l^d$. If 
\[\frac{1}{k}\sum_{d\mid k}\mu\left(\frac{k}{d}\right)l^d<\frac{n\cdot \phi(n)}{k}\] 
and the requirements on $\alpha$ remain the same, then $K$ is not monogenic over $\QQ$ by the same methods used above. One can also obtain weakened hypotheses on $\alpha$ via these ideas. 
\end{remark}

\begin{example}\label{5and11}

Consider $n=5$ and $l=11$. We see $11<5\cdot 4$. Since $11\equiv 1 \bmod 5$, the prime $11$ splits completely in $\QQ(\zeta_5)$. For $11$ to split completely in $\QQ\left(\zeta_5,\sqrt[5]{\alpha}\right)$, we need $\alpha$ to be a $5^{\text{th}}$ power in $\FF_{11}$. This is satisfied by rational integers congruent to $\pm1 \bmod {11}$. Hence all rational integers $\alpha\equiv \pm 1 \bmod {11}$ for which $x^5-\alpha$ is irreducible in $\ZZ(\zeta_5)[x]$ yield non-monogenic $K$. 

\end{example}

\section{General Radical Extensions}\label{genradsection}

In this section we consider an arbitrary number field $L$ and an element $\alpha\in \Ocal_L$ such that $x^n-\alpha$ is irreducible over $L$. To avoid trivialities, we assume $n\geq 2$. For a prime $\gp$ of $\Ocal_L$, we write $p$ for the residue characteristic and $f$ for the residue class degree. If $\gp$ divides $n$, we factor $n=p^em$ with $\gcd(m,p)=1$. Define $\varepsilon$ to be congruent to $e$ modulo $f$ with $1\leq \varepsilon\leq f$. For convenience, define $\beta$ to be $\alpha$ to the power $p^{f-\varepsilon}$. 
The Wieferich congruence, Congruence \eqref{Wieferich}, generalizes to
\begin{equation}\label{Wieferich3}
\alpha^{p^{f-\varepsilon+e}}=\beta^{p^e}\equiv \alpha \bmod {\gp^2}.
\end{equation}

In the case where $e\leq f$, this is
\[\alpha^{p^f}\equiv \alpha \bmod {\gp^2}.\]

\begin{theorem}\label{radical}
The ring of integers of $L\left(\sqrt[n]{\alpha}\right)$ is $\Ocal_L\left[\sqrt[n]{\alpha}\right]$ if and only if $\alpha$ is square-free as an ideal of $\Ocal_L$ and every prime $\gp$ dividing $n$ \textbf{does not} satisfy Congruence \eqref{Wieferich3}.
\end{theorem}

\begin{proof} We need only consider the prime divisors of $\Delta_{x^n-\alpha}=(-1)^\frac{n^2-n}{2}n^n(-\alpha)^{n-1}$. For any primes dividing $\alpha$, the argument is straightforward and essentially the same as in the proof of Theorem \ref{KoverQzeta}.

Maintaining the notation outlined above, let $\gp$ be a prime of $\Ocal_L$ dividing $n$, but not $\alpha$. Noting $\beta^{p^e}\equiv \alpha \bmod {\gp}$, we have

\[x^n-\alpha\equiv \big(x^m-\beta\big)^{p^e} \bmod {\gp}.\]
With the notation of Theorem \ref{Dedekindindex},

\[d(x)=\dfrac{x^n-\alpha- \left(x^m-\beta+a\pi_{\gp}\right)^{p^e}}{\pi_{\gp}},\]
where $a$ is some element of $\left(\Ocal_L\right)_{\gp}$ so that the term $a\pi_{\gp}$ accommodates possible further factorization of $x^m-\beta$ modulo $\gp$. 

The relative primality of $\overline{d(x)}$ and the factors of $\overline{x^m-\beta}$ does not change upon extension, so it suffices to work in $\left( \Ocal_L \right)_{\gp} \left( \sqrt[m]{\beta},\zeta_m \right)$ and $\left( \Ocal_L \right)_{\gp} \left(\sqrt[m]{\beta},\zeta_m \right)$ modulo $\left(\pi_{\gp}\right)$. With Theorem \ref{Dedekindindex} in mind, we wish to show that $\overline{d(x)}$ does not have $\overline{\zeta_m^k\sqrt[m]{\beta}}$ as a root for any $k$. Evaluating,

\[d\left(\zeta_m^k\sqrt[m]{\beta}\right)=\dfrac{\beta^{p^e}-\alpha-\left(\beta-\beta+a\pi_{\gp}\right)^{p^e}}{\pi_{\gp}}=\dfrac{\beta^{p^e}-\alpha +\left(a\pi_{\gp}\right)^{p^e}}{\pi_{\gp}}.\]
Clearly, $d\left(\zeta_m^k\sqrt[m]{\beta}\right)\equiv 0\bmod {\pi_{\gp}}$ if and only if 

\[\beta^{p^{e}}-\alpha=\alpha^{p^{f-\varepsilon+e}}-\alpha\equiv 0 \bmod {\pi_{\gp}^2}.\]
Our result follows.
\end{proof}

Combining the above proof with the generalization of Dedekind's criterion for splitting, \cite[Chapter I, Theorem 7.4]{Janusz}, we obtain
\begin{porism}\label{RadicalSplitting}
As above, suppose $x^n-\alpha\in \Ocal_L[x]$ is irreducible with $\alpha$ square-free. Let $\gp$ be a prime of $L$ that does not divide $n$. Then the splitting of $\gp$ in $L\left(\sqrt[n]{\alpha}\right)$ is mirrored by the splitting of $\overline{x^n-\alpha}$ in $\Ocal_L/\gp[x]$, as in Theorem \ref{DedCrit}. In particular, $\gp$ splits completely if and only if $\overline{\alpha}\neq 0$ is an $n^{\text{th}}$ root in $\Ocal_L/\gp$. Moreover, if Congruence \eqref{Wieferich3} holds, we can remove the restriction that $\gp\nmid n$.
\end{porism}

\bibliography{Bibliography}
\bibliographystyle{abbrvnat}

\end{document}